\documentclass[12pt,reqno]{amsart}

\usepackage[latin1]{inputenc}
\usepackage{amsmath}
\usepackage{amsfonts}
\usepackage{amssymb}
\usepackage{graphics}
\usepackage{enumerate}
\usepackage{amssymb,amsmath,amsthm,amscd,latexsym,verbatim,graphicx,amsfonts}

\usepackage{amscd}
\usepackage{amsmath}
\usepackage{amssymb}
\usepackage{amsthm}
\usepackage{latexsym}
\usepackage{verbatim}

\theoremstyle{plain}
\newtheorem{theorem}{Theorem}[section]

\newtheorem{lemma}[theorem]{Lemma}
\newtheorem{prop}[theorem]{Proposition}
\theoremstyle{definition}

\theoremstyle{remark}

\newcommand{\nri}{n\rightarrow\infty}

\newcommand{\bbZ}{\mathbb{Z}}
\newcommand{\bbR}{\mathbb{R}}
\newcommand{\bbC}{\mathbb{C}}

\newcommand{\bbN}{\mathbb{N}}

\DeclareMathOperator*{\Real}{Re}
\DeclareMathOperator*{\Imag}{Im}

\title[]{Convergence Rates of Exceptional Zeros of Exceptional Orthogonal Polynomials}
\author[]{Brian Simanek}
\date{}

\begin{document}
\maketitle

\begin{abstract}
We consider the zeros of exceptional orthogonal polynomials (XOP).  Exceptional orthogonal polynomials were originally discovered as eigenfunctions of second order differential operators that exist outside the classical Bochner-Brenke classification due to the fact that XOP sequences omit polynomials of certain degrees.  This omission causes several properties of the classical orthogonal polynomial sequences to not extend to the XOP sequences.  One such property is the restriction of the zeros to the convex hull of the support of the measure of orthogonality.  In the XOP case, the zeros that exist outside the classical intervals are called exceptional zeros and they often converge to easily identifiable limit points as the degree becomes large.  We deduce the exact rate of convergence and verify that certain estimates that previously appeared in the literature are sharp.
\end{abstract}

\vspace{4mm}


\vspace{2mm}


\vspace{2mm}

\normalsize

\section{Introduction}\label{intro}

The classical Bochner-Brenke classification states that (up to a linear change of variables) the only sequences of orthogonal polynomials $\{p_n\}_{n=0}^{\infty}$ with $\deg(p_n)=n$ that are a complete set of eigenfunctions for a second order linear ODE are the Jacobi polynomials, the Laguerre polynomials, and the Hermite polynomials \cite[Theorem 4.2.2]{SimonIV}.  In \cite{GKM09}, Gomez-Ullate, Kamran, and Milson initiated the study of exceptional orthogonal polynomial sequences by studying sequences of orthogonal polynomials that are a complete set of eigenfunctions of a second order linear differential operator, but for which the sequence omits polynomials of finitely many degrees.  The term exceptional orthogonal polynomials has since taken on a more general meaning (see \cite{Bonn,BK18}).   It was recently proven in \cite{GGM} that every exceptional orthogonal polynomial sequence is related to a classical orthogonal polynomial sequence by finitely many Darboux transformations.

The absence of polynomials of certain degrees from an exceptional orthogonal polynomial sequence forces these polynomials to differ from their classical counterparts in several important ways.  The particular property we will focus on is the location of the zeros.  The zeros of exceptional orthogonal polynomials fit into two categories: regular and exceptional.  The regular zeros are the ones that exist in the support of the classical measure of orthogonality and the exceptional zeros are the ones that exist outside this interval.  For example, the regular zeros of an exceptional Hermite polynomial are the real zeros and the the exceptional zeros of an exceptional Hermite polynomial are the non-real zeros.  Under very mild hypotheses, one can understand the behavior of the exceptional zeros and show that the number of exceptional zeros is bounded as the degree of the polynomial tends to infinity and the exceptional zeros tend to easily identifiable limit points as the degree tends to infinity.  Our main results are estimates on the rate of convergence of these exceptional zeros to their limit points.

Such rates of convergence have been studied before.  In \cite{Bonn} Bonneux showed that in the case of exceptional Jacobi polynomials, exceptional zeros approach their limit points at a rate of $O(n^{-1})$ as $\nri$, where $n$ is the degree of the polynomial (see also \cite{DL,Lun}).  In \cite{BK18} Bonneux and Kuijlaars showed that in the case of exceptional Laguerre polynomials, exceptional zeros approach their limit points at a rate of $O(n^{-1/2})$ as $\nri$.  In \cite{KM} Kuijlaars and Milson showed that in the case of exceptional Hermite polynomials, exceptional zeros approach their limit points at a rate of $O(n^{-1/2})$ as $\nri$.  We will show that all three estimates are - in general - sharp and provide an exact formula for the leading order rate of convergence.

In Section \ref{back} we will state our new results and review some of the necessary background on the subject of exceptional orthogonal polynomial sequences and other results from the literature that will be necessary to prove our main results.  In Section \ref{jac} we will prove our results for exceptional Jacobi polynomials.  In Section \ref{lag} we will prove our results for exceptional Laguerre polynomials.  In Section \ref{herm} we will prove our results for exceptional Hermite polynomials.  In Section \ref{examples} we will present some examples that illustrate the estimates from our theorems.

\section{Background $\&$ Results}\label{back}

In this section, we will review some of the notation, terminology, and results from the literature that will be needed in our later analysis.

\subsection{Partitions}\label{part}

If $r\in\bbN\cup\{0\}$, then a partition $\lambda$ will be an $r$-tuple of non-increasing natural numbers:
\begin{align*}
\lambda&=(\lambda_1,\ldots,\lambda_{r}),\qquad\qquad\lambda_1\geq\lambda_2\cdots\geq\lambda_{r}\geq1.
\end{align*}
We also define $|\lambda|$ by the formula
\[
|\lambda|=\lambda_1+\cdots+\lambda_{r}.
\]
We call the partition $\lambda$ an \textit{even} partition if $r$ is even and $\lambda_{2j}=\lambda_{2j-1}$ for $j=1,\ldots,r/2$.  For completeness, we note that the empty partition (for which $r=0$) is an even partition by convention.

\subsection{Exceptional Jacobi Polynomials}\label{xjac}

The exceptional Jacobi polynomials are generalizations of the classical Jacobi polynomials that are orthogonal on $[-1,1]$ with an appropriate weight.  We will be interested in studying properties of their zeros.  In \cite[Theorem 6.6]{Bonn}, it was shown that the exceptional zeros of certain exceptional Jacobi polynomials tend to their limit points at a rate of $O(n^{-1})$ as $\nri$, where $n$ is the degree of the polynomial.  We will show that this result is sharp and provide the leading order of the asymptotics in many cases.

We begin our discussion with a review of some important properties of classical and generalized Jacobi polynomials and then discuss the exceptional Jacobi polynomials.

\subsubsection{Classical Jacobi Polynomials}\label{classjac}

Recall that for any $\alpha,\beta>-1$, the classical Jacobi polynomials $\{P_n^{(\alpha,\beta)}\}_{n=0}^{\infty}$ are orthogonal on the interval $[-1,1]$ with respect to the weight function $(1-x)^{\alpha}(1+x)^{\beta}$.  The polynomial $P_n^{(\alpha,\beta)}$ has degree $n$ and satisfies the second order linear ODE
\[
(1-x^2)y''+(\beta-\alpha-(\alpha+\beta+2)x)y'+n(n+\alpha+\beta+1)y=0.
\] 
The ODE only defines the polynomial $P_n^{(\alpha,\beta)}$ up to a multiplicative constant, so we will write explicitly
\[
P_n^{(\alpha,\beta)}(z)=\frac{\Gamma(\alpha+n+1)}{n!\Gamma(\alpha+\beta+n+1)}\sum_{m=0}^n\binom{n}{m}\frac{\Gamma(\alpha+\beta+n+m+1)}{\Gamma(\alpha+m+1)}\left(\frac{z-1}{2}\right)^m
\]
Notice that one can make sense of this expression for any complex numbers $\alpha$ and $\beta$ so we will take this as our definition of $P_n^{(\alpha,\beta)}(z)$ for all $\alpha,\beta\in\bbC$.  Only the case $\alpha,\beta>-1$ corresponds to orthogonal polynomials (but see \cite{KMO}).  
From this formula for $P_n^{(\alpha,\beta)}(z)$, one can check that
\begin{equation}\label{jacdiff}
\frac{d}{dz}P_n^{(\alpha,\beta)}(z)=\frac{\alpha+\beta+n+1}{2}P_{n-1}^{(\alpha,+1\beta+1)}(z)
\end{equation}

Our analysis will also require some detailed knowledge of generalized Jacobi polynomials, which we now define following the notation and formulas from \cite[Section 2]{Bonn}.  To do so, we must first fix two partitions $\lambda$ and $\mu$:
\begin{align*}
\lambda&=(\lambda_1,\ldots,\lambda_{r_1}),\qquad\lambda_1\geq\lambda_2\cdots\geq\lambda_{r_1}\geq1\\
\mu&=(\mu_1,\ldots,\mu_{r_2}),\qquad\mu_1\geq\mu_2\cdots\geq\mu_{r_2}\geq1
\end{align*}
If we set $r=r_1+r_2$, then we can define the generalized Jacobi polynomial $\Omega_{\lambda,\mu}^{(\alpha,\beta)}$ by the Wronskian determinant
\[
\Omega_{\lambda,\mu}^{(\alpha,\beta)}(x)=(1+x)^{(\beta+r_1)r_2}\cdot\mbox{Wr}[f_1,\ldots,f_r]
\]
where
\begin{align*}
f_j(x)&=P_{\lambda_j+r_1-j}^{(\alpha,\beta)}(x),\qquad\qquad\qquad\qquad j=1,\ldots,r_1\\
f_{r_1+j}(x)&=(1+x)^{-\beta}P_{\mu_j+r_2-j}^{(\alpha,-\beta)}(x),\qquad\qquad j=1,\ldots,r_2
\end{align*}
According to \cite[Lemma 2.3]{Bonn} and \cite{D17}, $\Omega_{\lambda,\mu}^{(\alpha,\beta)}(x)$ is a polynomial of degree exactly $|\lambda|+|\mu|$ as long as all of the following conditions are satisfied:
\begin{itemize}
\item[(C1)] $\alpha+\beta+\lambda_j+r_1-j\not\in\{-1,-2,\ldots,-(\lambda_j+r_1-j)\}$ for all $j=1,\ldots,r_1$
\item[(C2)] $\alpha-\beta+\mu_j+r_2-j\not\in\{-1,-2,\ldots,-(\mu_j+r_2-j)\}$ for all $j=1,\ldots,r_2$
\item[(C3)] $\beta\neq(\mu_j+r_2-j)-(\lambda_i+r_1-i)$ for all $i=1,\ldots,r_1$ and $j=1,\ldots,r_2$
\end{itemize}
(see also \cite{D15}).  According to \cite{D17} and \cite[Lemma 2.8]{Bonn}, if $\alpha>-1$, $\beta>\mu_1+r_2-1$, and (C1-C3) are satisfied, then the polynomial $\Omega_{\lambda,\mu}^{(\alpha,\beta)}(x)$ has no zeros in $[-1,1]$ if $\lambda$ is an even partition.

\subsubsection{Exceptional Jacobi Polynomials}\label{xmjac}

Now we will show how to construct the exceptional Jacobi polynomials by following the construction in \cite{Bonn} (the special case of Legendre polynomials is discussed from a different point of view in \cite{GGM2}).  Let us retain the notation from the previous subsection.  Let us also define
\[
\bbN_{\lambda,\mu}=\{n\in\bbN\cup\{0\}:n\geq|\lambda|+|\mu|-r_1,\, n-|\lambda|-|\mu|\neq\lambda_j-j\,\,\mbox{for}\,\, j=1,2,\ldots,r_1\}.
\]
It is easy to see that $|\bbN_0\setminus\bbN_{\lambda,\mu}|=|\lambda|+|\mu|$, where $\bbN_0=\bbN\cup\{0\}$.  If $\alpha,\beta\in\bbR$ satisfy conditions (C1-C3) from the previous subsection and $n\in\bbN_{\lambda,\mu}$, one can define the exceptional Jacobi polynomial $P_{\lambda,\mu,n}^{(\alpha,\beta)}$ by
\[
P_{\lambda,\mu,n}^{(\alpha,\beta)}(x)=(1+x)^{(\beta+r_1+1)r_2}\cdot\mbox{Wr}[f_1,\ldots,f_r,P_s^{(\alpha,\beta)}(x)],
\]
where $s=n-|\lambda|-|\mu|+r_1$ (see \cite[Definition 2.11]{Bonn}).  We now recall the following result, which comes from \cite[Lemma 2.13]{Bonn}.

\begin{lemma}[\cite{Bonn}]\label{new213}
Assume $\lambda$ and $\mu$ are partitions and that $\alpha,\beta\in\bbR$ satisfy (C1-C3) above.  Assume further that
\begin{align*}
&\alpha+\beta+n-|\lambda|-|\mu|+r_1\not\in\{-1,-2,\ldots,-(n-|\lambda|-|\mu|+r_1)\}\\
&\beta\neq(\mu_j+r_2-j)-(n-|\lambda|-|\mu|+r_1)\qquad\qquad\qquad\qquad\qquad\qquad\qquad j=1,\ldots,r_2.
\end{align*}
Then the degree of $P_{\lambda,\mu,n}^{(\alpha,\beta)}$ is exactly $n$ when $n\in\bbN_{\lambda,\mu}$.
\end{lemma}

We note here that the $X_m$ exceptional Jacobi polynomials studied in \cite{DL,GMM,HS,H15,LLS,Lun} are special cases of this more general construction (see \cite[Section 5.2]{Bonn}).

The zeros of $P_{\lambda,\mu,n}^{(\alpha,\beta)}$ in $(-1,1)$ are called \textit{regular zeros}.  The following result was shown in \cite[Section 6.1]{Bonn}.

\begin{prop}[\cite{Bonn}]\label{new61}
If $\alpha>-1$, $\beta>\mu_1+r_2-1$, $\lambda$ is an even partition, and $\alpha$ and $\beta$ satisfy the conditions of Lemma \ref{new213} for all $n\in\bbN_{\lambda,\mu}$, then $P_{\lambda,\mu,n}^{(\alpha,\beta)}$ has exactly $n-|\lambda|-|\mu|$ simple regular zeros when $n$ is large.  The remaining $|\lambda|+|\mu|$ zeros of $P_{\lambda,\mu,n}^{(\alpha,\beta)}$ are called the \textit{exceptional zeros}.
\end{prop}

If the hypotheses of Proposition \ref{new61} are satisfied, let us denote the regular zeros of $P_{\lambda,\mu,n}^{(\alpha,\beta)}(z)$ by $z_{1,n}^{(\alpha,\beta)},z_{2,n}^{(\alpha,\beta)},\ldots,z_{n-|\lambda|-|\mu|,n}^{(\alpha,\beta)}$ and the exceptional zeros by $\zeta_{1,n}^{(\alpha,\beta)},\ldots,\zeta_{|\lambda|+|\mu|,n}^{(\alpha,\beta)}$.  According to \cite[Theorem 6.6]{Bonn}, if the zeros of $\Omega_{\lambda,\mu}^{(\alpha,\beta)}$ are simple, then one can index the set $\{\zeta_{j,n}^{(\alpha,\beta)}\}_{j=1}^{|\lambda|+|\mu|}$ so that for each $1\leq k\leq |\lambda|+|\mu|$ there exists a complex number $\zeta_{k,\infty}^{(\alpha,\beta)}$ such that $\Omega_{\lambda,\mu}^{(\alpha,\beta)}(\zeta_{k,\infty}^{(\alpha,\beta)})=0$ and $\zeta_{k,n}^{(\alpha,\beta)}\rightarrow\zeta_{k,\infty}^{(\alpha,\beta)}$ as $\nri$.  It was shown in that same theorem that $|\zeta_{k,n}^{(\alpha,\beta)}-\zeta_{k,\infty}^{(\alpha,\beta)}|=O(n^{-1})$ as $\nri$, and it was also mentioned that this estimate is not known to be sharp \cite[Remark 6.7]{Bonn} (see also \cite{Lun}).  Our main result is the following theorem that shows that the estimate from \cite{Bonn} is sharp when $\lambda$ is even and we calculate the precise leading order asymptotics.

\begin{theorem}\label{jacthm}
If the hypotheses of Proposition \ref{new61} are satisfied and the zeros of $\Omega_{\lambda,\mu}^{(\alpha,\beta)}$ are simple, then
\[
\lim_{\nri}n[\zeta_{k,n}^{(\alpha,\beta)}-\zeta_{k,\infty}^{(\alpha,\beta)}]=\sqrt{\left(\zeta_{k,\infty}^{(\alpha,\beta)}\right)^2-1},
\]
where we take the branch of the square root that is positive on $[1,\infty)$ and with branch cut $[-1,1]$.
\end{theorem}

The assumption that the zeros of $\Omega_{\lambda,\mu}^{(\alpha,\beta)}$ are simple is straightforward to check numerically for any specific choice of $\alpha$ and $\beta$ satisfying the hypotheses of Theorem \ref{jacthm}.  Of course it is always true when $|\lambda|+|\mu|=1$ (which would force $\lambda=\emptyset$ if $\lambda$ is even).

\subsection{Exceptional Laguerre Polynomials}\label{xlag}

The exceptional Laguerre polynomials are generalizations of the classical Laguerre polynomials that are orthogonal on $[0,\infty)$ with an appropriate weight.  We will be interested in studying properties of their zeros.  In \cite[Theorem 5]{BK18}, it was shown that the exceptional zeros of certain exceptional Laguerre polynomials tend to their limit points at a rate of $O(n^{-1/2})$ as $\nri$, where $n$ is the degree of the polynomial.  We will show that this result is sharp and provide the leading order of the asymptotics in many cases.

We begin our discussion with a review of some important properties of classical and generalized Laguerre polynomials and then discuss the exceptional Laguerre polynomials.

\subsubsection{Classical and Generalized Laguerre Polynomials}\label{lagclass}

For any $\alpha\in(-1,\infty)$, the Laguerre polynomials $\{L_n^{(\alpha)}\}_{n=0}^{\infty}$ are an orthogonal sequence with respect to the weight $x^{\alpha}e^{-x}$ on the interval $(0,\infty)$.  The degree $n$ polynomial $L_n^{(\alpha)}(x)$ satisfies the ODE
\[
xy''+(\alpha+1-x)y'+ny=0.
\]
This ODE only defines each Laguerre polynomial up to a multiplicative constant, so we will specify
\[
L_n^{(\alpha)}(z)=\sum_{j=0}^n\binom{n+\alpha}{n-j}\frac{(-z)^j}{j!},
\]
which allows us to define $L_n^{(\alpha)}(z)$ for all complex numbers $\alpha$.

There are several properties of the Laguerre polynomials that will be relevant to our analysis.  The first is the Hahn property that
\begin{equation}\label{hahn}
\frac{d}{dx}L_n^{(\alpha)}(x)=-L_{n-1}^{(\alpha+1)}(x),\qquad n\geq1.
\end{equation}
The second is a ratio asymptotic result from \cite{DHM}, which we will shortly state precisely.  We should mention that the result \cite[Theorem 3]{DHM} is much stronger than what we will state here, which is a consequence of that result that is strong enough for our calculations.

\begin{theorem}\label{dhmthm}[\cite{DHM}]
If $\alpha,\beta>-1$ are fixed, $j\in\bbZ$ is fixed, and $z\in\bbC\setminus[0,\infty)$ is fixed, then
\[
\frac{L_{n+j}^{(\alpha)}(z)}{L_{n}^{(\beta)}(z)}=\left(-\frac{z}{n}\right)^{\frac{\beta-\alpha}{2}}\left(1+O(n^{-1/2})\right)+O\left(\frac{1}{n}\right)
\]
as $\nri$.  The error terms hold uniformly for $z$ in compact subsets of $\bbC\setminus[0,\infty)$.
\end{theorem}

We will also need some understanding of the zeros of $L_n^{(\alpha)}$.  These zeros are all in $(0,\infty)$, are simple, and the zeros of $L_n^{(\alpha)}$ interlace those of $L_{n+1}^{(\alpha)}$.  As $\nri$, the zeros tend to accumulate to infinity, meaning that any fixed compact subset of $[0,\infty)$ contains only $o(n)$ zeros of $L_n^{(\alpha)}$ as $\nri$.  In fact, using results from \cite{G02}, one can show that any fixed compact subset of $[0,\infty)$ contains at most $O(\sqrt{n})$ zeros of $L_n^{(\alpha)}$ as $\nri$.

Our analysis will also require some detailed knowledge of generalized Laguerre polynomials, which we now define.  To do so, we must first fix two partitions $\lambda=(\lambda_1,\ldots,\lambda_{r_1})$ and $\mu=(\mu_1,\ldots,\mu_{r_2})$ as in the construction of the generalized Jacobi polynomials.  If we set $r=r_1+r_2$, then we can define the generalized Laguerre polynomial $\Omega_{\lambda,\mu}^{(\alpha)}$ by the Wronskian determinant
\[
\Omega_{\lambda,\mu}^{(\alpha)}(x)=e^{-r_2x}\cdot\mbox{Wr}[f_1,\ldots,f_r]
\]
where
\begin{align*}
f_j(x)&=L_{\lambda_j+r_1-j}^{(\alpha)}(x),\qquad\qquad j=1,\ldots,r_1\\
f_{r_1+j}(x)&=e^xL_{\mu_j+r_2-j}^{(\alpha)}(-x),\qquad\qquad j=1,\ldots,r_2
\end{align*}
According to \cite[Lemma 1]{BK18}, $\Omega_{\lambda,\mu}^{(\alpha)}(x)$ is a polynomial of degree exactly $|\lambda|+|\mu|$.  According to \cite{D14a,DP}, if $\alpha>-1$, then the polynomial $\Omega_{\lambda,\mu}^{(\alpha)}(x)$ has no zeros in $[0,\infty)$ if and only if $\lambda$ is an even partition.

\subsubsection{Exceptional Laguerre Polynomials}\label{xlagintro}

Now we will show how to construct the exceptional Laguerre polynomials by following the construction in \cite{BK18}.  Let us retain the notation from the previous subsection.  Let us also define $\bbN_{\lambda,\mu}$ as in the construction of the exceptional Jacobi polynomials.  If $\alpha\in\bbR$ and $n\in\bbN_{\lambda,\mu}$, one can define the exceptional Laguerre polynomial $L_{\lambda,\mu,n}^{(\alpha)}$ by
\[
L_{\lambda,\mu,n}^{(\alpha)}(x)=e^{-r_2x}\cdot\mbox{Wr}[f_1,\ldots,f_r,L_s^{(\alpha)}(x)],
\]
where $s=n-|\lambda|-|\mu|+r_1$ (see \cite[Definition 4]{BK18}).  By construction, the degree of $L_{\lambda,\mu,n}^{(\alpha)}$ is exactly $n$.  We note here that the $X_m$ exceptional Laguerre polynomials of Type-$I$, Type-$II$, and Type-$III$ studied in \cite{GMM,HS,H18,LLMS,Lun} are special cases of this more general construction (see \cite[Proposition 4]{BK18}).

We will be interested in the zeros of $L_{\lambda,\mu,n}^{(\alpha)}$.  It was shown in \cite[Section 6.1]{BK18} that if $\alpha>-1$ and $\lambda$ is even, then the polynomial $L_{\lambda,\mu,n}^{(\alpha)}$ has $n-|\lambda|-|\mu|$ simple zeros in the interval $(0,\infty)$ when $n$ is large (we call these zeros \textit{regular zeros}).  It was shown in \cite[Section 6.4]{BK18} that if the zeros of $\Omega_{\lambda,\mu}^{(\alpha)}$ are simple and $\lambda$ is even, then there are exactly $|\lambda|+|\mu|$ zeros of $L_{\lambda,\mu,n}^{(\alpha)}$ in $\bbC\setminus[0,\infty)$ when $n$ is large (we call these zeros \textit{exceptional zeros}).  Concerning the regular zeros, an even more precise result was given in \cite[Corollary 2]{BK18} (which depends on a result from \cite[Section 3]{BD}).  That corollary states that if $n$ is large and $\{\sigma_{j,n}^{(\alpha+r)}\}_{j=1}^n$ are the zeros of $L_n^{(\alpha+r)}$ in increasing order, then at least $n-2|\lambda|-2|\mu|-r_2$ intervals $(\sigma_{j,n}^{(\alpha+r)},\sigma_{j+1,n}^{(\alpha+r)})$ contain a zero of $L_{\lambda,\mu,n}^{(\alpha)}$.

The exceptional zeros of $L_{\lambda,\mu,n}^{(\alpha)}$ also behave predictably when $\alpha>-1$, $\lambda$ is even, and the zeros of $\Omega_{\lambda,\mu}^{(\alpha)}$ are simple.  If we denote these exceptional zeros by $\{\zeta_{k,n}^{(\alpha)}\}_{k=1}^{|\lambda|+|\mu|}$, then it was shown in \cite[Theorem 5]{BK18} that one can index the exceptional zeros so that for each $k=1,\ldots,m$ there is a $\zeta_{k,\infty}^{(\alpha)}$ such that $\Omega_{\lambda,\mu}^{(\alpha)}(\zeta_{k,\infty}^{(\alpha)})=0$ and $\zeta_{k,n}^{(\alpha)}\rightarrow\zeta_{k,\infty}^{(\alpha)}$ at the rate $O(n^{-1/2})$ as $\nri$.  In \cite[Remark 6.7]{Bonn} it was mentioned that this estimate is not known to be sharp.  Our main result is the following theorem that shows that the estimate from \cite{BK18} is sharp when $\lambda$ is even and we calculate the precise leading order asymptotics.

\begin{theorem}\label{lagthm}
If $\alpha>-1$, $\lambda$ is even, and the zeros of $\Omega_{\lambda,\mu}^{(\alpha)}$ are simple, then
\[
\lim_{\nri}\sqrt{n}[\zeta_{k,n}^{(\alpha)}-\zeta_{k,\infty}^{(\alpha)}]=-\sqrt{-\zeta_{k,\infty}^{(\alpha)}},
\]
where we take the branch of the square root that is positive on $(0,\infty)$ and with branch cut $(-\infty,0]$.
\end{theorem}

\noindent\textit{Remark.}  The assumption of simple zeros in Theorem \ref{lagthm} is not a strong one.  Indeed, it is conjectured in \cite{BK18} that the other assumptions in Theorem \ref{lagthm} are sufficient to guarantee that the zeros of $\Omega_{\lambda,\mu}^{(\alpha)}$ are simple (see \cite[Conjecture 1]{BK18}).


\medskip

The first part of \cite[Proposition 4]{BK18} tell us that Theorem \ref{lagthm} applies to the Type-$I$ exceptional Laguerre polynomials that were studied in \cite{GMM,LLMS,Lun} (the simplicity of the zeros of $\Omega_{\lambda,\mu}^{(\alpha)}$ in this case follows from \cite[Proposition 3.4]{GMM}).  The second part of \cite[Proposition 4]{BK18} tell us that Theorem \ref{lagthm} applies to the Type-$II$ exceptional Laguerre polynomials that were studied in \cite{GMM,LLMS} as long as the zeros of $L_m^{(-\alpha-1)}$ are simple for $\alpha>m-1$.  The third part of \cite[Proposition 4]{BK18} tells us that the Type-$III$ exceptional Laguerre polynomials introduced in \cite{LLMS} need not satisfy the hypotheses of Theorem \ref{lagthm} because the partition $\lambda$ need not be even.  Nevertheless, a comparable result does hold.

Fix some $m\in\bbN$ and $\alpha\in(-1,0)$.  The Type-$III$ exceptional Laguerre polynomials are defined by
\[
L_{m,n}^{III(\alpha)}(x)=-nL_{(1,\ldots,1),\emptyset,n}^{(\alpha-m)}(x),\qquad n>m
\]
(where $r_1=m$) and $L_{m,0}^{III(\alpha)}(x)=1$ (see \cite[Proposition 4]{BK18}).  If $m$ is not even, then the partition $(1,\ldots,1)$ is not an even partition.

The polynomial $L_{m,n}^{III(\alpha)}(x)$ is a solution of the ODE
\begin{align}\label{xlag3ode}
y''+\left(\frac{\alpha+1-x}{x}-\frac{2L_m^{(-\alpha-1)}(-x)'}{L_m^{(-\alpha-1)}(-x)}\right)y'+\left(\frac{n}{x}\right)y=0
\end{align}
If $n>m$, the polynomial $L_{m,n}^{III(\alpha)}$ has $n-m$ simple zeros in the interval $(0,\infty)$ (the regular zeros) and $m$ simple zeros in $(-\infty,0)$ (the exceptional zeros) (see \cite{LLMS}).  To be more precise, the regular zeros of $L_{m,n}^{III(\alpha)}$ interlace the zeros of $L_{n-m-1}^{(\alpha+1)}$ (see \cite[Theorem 5.5]{LLMS}).  This interlacing property tells us that the asymptotic distribution of the regular zeros of $L_{m,n}^{III(\alpha)}$ for large $n$ is very similar to the distribution of zeros of classical Laguerre polynomials.

The exceptional zeros of $L_{m,n}^{III(\alpha)}$ also behave predictably.  If we denote these zeros in increasing order by $\{\zeta_{k,n}^{III(\alpha)}\}_{k=1}^m$, then it was shown in \cite[Theorem 5.6]{LLMS} that for each $k=1,\ldots,m$ there is a negative real number $\zeta_{k,\infty}^{III(\alpha)}$ such that $L_m^{(-\alpha-1)}(-\zeta_{k,\infty}^{III(\alpha)})=0$ and it follows from \cite[Theorem 5]{BK18} that $|\zeta_{k,n}^{III(\alpha)}-\zeta_{k,\infty}^{III(\alpha)}|=O(n^{-1/2})$ as $\nri$.  We can extend Theorem \ref{lagthm} to this additional case to show that the estimate from \cite[Theorem 5]{BK18} is sharp in this case as well.

\begin{theorem}\label{lag3thm}
It holds that
\[
\lim_{\nri}\sqrt{n}[\zeta_{k,n}^{III(\alpha)}-\zeta_{k,\infty}^{III(\alpha)}]=-\sqrt{-\zeta_{k,\infty}^{III(\alpha)}},
\]
where we take the branch of the square root that is positive on $(0,\infty)$ and with branch cut $(-\infty,0]$.
\end{theorem}

\subsection{Exceptional Hermite Polynomials}\label{xhermpoly}

The exceptional Hermite polynomials were introduced in \cite{GGM14}, where they were used to find new examples of exactly solvable potentials for the Schr\"{o}dinger equation.  They have been studied in greater detail in \cite{D14,GKKM,KM,KM20}.  In \cite[Theorem 2.3]{KM}, it was shown that (under appropriate hypotheses) the exceptional zeros of exceptional Hermite polynomials tend to their limit points at a rate of $O(n^{-1/2})$ as $\nri$, where $n$ is the degree of the polynomial.  We will show that this result is sharp.

We begin our discussion with a review of some important properties of classical Hermite polynomials and their generalizations.  Then we will discuss the properties of the exceptional Hermite polynomials that we will need in our analysis.

\subsubsection{Classical and Generalized Hermite Polynomials}\label{hermclass}

The sequence of Hermite polynomials $\{H_n(x)\}_{n=0}^{\infty}$ is orthogonal on $\bbR$ with respect to the weight function $e^{-x^2}$.  The orthogonality only defines each polynomial up to a multiplicative constant, so we will choose this constant so that the following relations hold:
\begin{equation}\label{hermlag}
H_{2n}(x)=(-4)^nn!L_n^{(-1/2)}(x^2),\qquad H_{2n+1}(x)=2(-4)^nn!xL_n^{(1/2)}(x^2)
\end{equation}
From these formulas, one can deduce many properties of $H_n$ such as its asymptotic zero distribution and the fact that
\[
\frac{d}{dx}H_n(x)=2nH_{n-1}(x),\qquad n\geq1.
\]

Our analysis will require the use of generalized Hermite polynomials, which are defined using formulas involving the Wronskian determinant of Hermite polynomials.  Take some $r\in\bbN$ and let $\lambda=\{\lambda_j\}_{j=1}^r$ be a partition.  Associated to the partition $\lambda$ is the polynomial
\[
H_{\lambda}:=\mbox{Wr}[H_{\lambda_r},H_{\lambda_{r-1}+1},\ldots,H_{\lambda_1+r-1}],
\]
whose degree is $|\lambda|$ (see \cite[Section 2]{KM}).  If $\lambda$ is an even partition, then $H_{\lambda}$ has no real zeros (see \cite[Section 2]{KM}).

\subsubsection{Exceptional Hermite Polynomials}\label{hermx}

Using the notation above, suppose $\lambda$ is given as in the previous section and $m=|\lambda|$.  Define $\bbN_{\lambda}$ as in \cite{KM} by
\[
\bbN_{\lambda}=\{n:n\geq m-r,n\neq m-j+\lambda_j\,\mbox{for}\, j=1,2,\ldots,r\}.
\]
It is easy to see that the cardinality of $\bbN_0\setminus\bbN_{\lambda}$ is precisely $m$.  If $n\in\bbN_{\lambda}$, we may define
\[
H_{m,n}^{(\lambda)}(x)=\mbox{Wr}[H_{\lambda_r},H_{\lambda_{r-1}+1},\ldots,H_{\lambda_1+r-1},H_{n-m+r}]
\]
The polynomial $H_{m,n}^{(\lambda)}$ has degree $n$ (since $n\in\bbN_{\lambda}$) and we will call the sequence $\{H_{m,n}^{(\lambda)}\}_{n\in\bbN_{\lambda}}$ the sequence of $X_m$ exceptional Hermite polynomials with partition $\lambda$.

The polynomial $H_{m,n}^{(\lambda)}(x)$ satisfies the following ODE:
\begin{align}\label{xhermode}
y''-2\left(x+\frac{H_{\lambda}(x)'}{H_{\lambda}(x)}\right)y'+\left(\frac{H_{\lambda}(x)''}{H_{\lambda}(x)}+2x\frac{H_{\lambda}(x)'}{H_{\lambda}(x)}+2n-2m\right)y=0
\end{align}

We will be interested in the zeros of $H_{m,n}^{(\lambda)}$.  If $n\geq m+\lambda_1$, the polynomial $H_{m,n}^{(\lambda)}$ has $n-m$ real zeros (we call these zeros \textit{regular zeros}) and $m$ non-real zeros (we call these zeros \textit{exceptional zeros}) \cite[Section 2]{KM}.  To be more precise, when $\lambda$ is an even partition the regular zeros of $H_{m,n}^{(\lambda)}$ are situated such that at least $n-m-r$ open intervals between consecutive zeros of  $H_{n}$ contain a zero of $H_{m,n}^{(\lambda)}$ (see \cite[Corollary 4.3]{KM}).  This interlacing property tells us that the asymptotic distribution of the regular zeros of $H_{m,n}^{(\lambda)}$ for large $n$ is very similar to the distribution of zeros of classical Hermite polynomials.

The exceptional zeros $\{\zeta_{k,n}^{(\lambda)}\}_{k=1}^m$ of $H_{m,n}^{(\lambda)}$ also behave predictably.  It was shown in \cite[Theorem 2.3]{KM} that if the zeros $\{\zeta_{k,\infty}^{(\lambda)}\}_{k=1}^m$ of $H_{\lambda}$ are simple and $\lambda$ is an even partition, then one can index the set $\{\zeta_{k,n}^{(\lambda)}\}_{k=1}^m$ so that $|\zeta_{k,n}^{(\lambda)}-\zeta_{k,\infty}^{(\lambda)}|=O(n^{-1/2})$ as $\nri$, though it was not shown that this estimate is sharp (but see \cite[Lemma 6]{H17}).  Our main result is the following theorem that shows that the estimate from \cite{KM} is sharp and we calculate the precise leading order asymptotics.

\begin{theorem}\label{hermthm}
If $\lambda$ is an even partition and the zeros of $H_{\lambda}$ are all simple, it holds that
\[
\lim_{\nri}\sqrt{n}[\zeta_{k,n}^{(\lambda)}-\zeta_{k,\infty}^{(\lambda)}]=\frac{\zeta_{k,\infty}}{\sqrt{-2\left(\zeta_{k,\infty}^{(\lambda)}\right)^2}},
\]
where we take the branch of the square root that is positive on $(0,\infty)$ and with branch cut $(-\infty,0]$.
\end{theorem}

The assumption that the zeros of $H_{\lambda}$ are simple is not a strong one and it was conjectured in \cite{FHV} that the non-zero zeros of $H_{\lambda}$ are always simple (even if $\lambda$ is not even).  Recent results stemming from an effort to prove this conjecture can be found in \cite{Grosux2}.

\section{Proof of Theorem \ref{jacthm}}\label{jac}

For this section, we will retain all of the notation from Section \ref{xjac}.  We begin with the following lemma.

\begin{lemma}\label{jacpole}
Under the assumptions of Theorem \ref{jacthm}, the polynomial $P_{\lambda,\mu,n}^{(\alpha,\beta)}$ does not vanish at any of the zeros of $\Omega_{\lambda,\mu}^{(\alpha,\beta)}$ when $n$ is large.
\end{lemma}

\begin{proof}
We will rely on calculations from the proof of \cite[Lemma 7.6]{Bonn}.  In particular, it is shown there that the function
\[
F_n(x):=\sin(x)^{\alpha+r+1/2}\cos(x)^{\beta+r+1/2}\cdot\frac{P_{\lambda,\mu,n}^{(\alpha,\beta)}(\cos(2x))}{\Omega_{\lambda,\mu}^{(\alpha,\beta)}(\cos(2x))}
\]
is an eigenfunction of the differential operator $S_{\lambda,\mu}[y]=-y''+V_{\lambda,\mu}y$, where
\begin{align*}
&V_{\lambda,\mu}(x)=\frac{(\alpha+2r-1/2)(\alpha+1/2)}{\sin^2(x)}+\frac{(\beta+2r-1/2)(\beta+1/2)}{\cos^2(x)}+\frac{r^2-r}{\sin^2(x)\cos^2(x)}\\
&\hspace{4in}-2\frac{d^2}{dx^2}\log\left(\Omega_{\lambda,\mu}^{(\alpha,\beta)}(\cos(2x))\right)
\end{align*}
From our assumptions about the zeros of $\Omega_{\lambda,\mu}^{(\alpha)}(x)$, it follows that $V_{\lambda,\mu}(x)$ has a pole of order $2$ at any zero of $\Omega_{\lambda,\mu}^{(\alpha,\beta)}(\cos(2x))$ when $n$ is large.  However, if $F_n(x)$ does not have a pole at a particular zero of $\Omega_{\lambda,\mu}^{(\alpha,\beta)}(\cos(2x))$, then neither does $F_n''(x)$ and hence $F_n(x)$ cannot be an eigenfunction of $S_{\lambda,\mu}$.  This contradiction proves the result.
\end{proof}

We also need the following lemma, which is a consequence of \cite[Theorem 6.5]{Bonn}.

\begin{lemma}\label{xjacdist}
Under the assumptions of Theorem \ref{jacthm}, for each $k=1,\ldots,|\lambda|+|\mu|$ it holds that
\[
\lim_{\nri}\frac{1}{n}\sum_{j=1}^{n-|\lambda|-|\mu|}\frac{1}{\zeta_{k,\infty}^{(\alpha,\beta)}-z_{j,n}^{(\alpha,\beta)}}=\sqrt{\frac{1}{\left(\zeta_{k,\infty}^{(\alpha,\beta)}\right)^2-1}},
\]
where we take the branch of the square root that is positive on $[1,\infty)$ and with branch cut $[-1,1]$.
\end{lemma}


\begin{proof}[Proof of Theorem \ref{jacthm}]
By Lemma \ref{jacpole}, we know that $\zeta_{k,n}^{(\alpha,\beta)}\neq\zeta_{k,\infty}^{(\alpha,\beta)}$ when $n$ is large.  Therefore, we may use \cite[Equation 7.20]{Bonn}, which tells us
\begin{align*}
&\sum_{j=1}^{n-|\lambda|-|\mu|}\frac{1}{\zeta_{k,\infty}^{(\alpha,\beta)}-z_{j,n}^{(\alpha,\beta)}}+\sum_{j=1}^{|\lambda|+|\mu|}\frac{1}{\zeta_{k,\infty}^{(\alpha,\beta)}-\zeta_{j,n}^{(\alpha,\beta)}}\\
&\hspace{1in}=\frac{\alpha+r}{2(1-\zeta_{k,\infty}^{(\alpha,\beta)})}-\frac{\beta+r}{2(1+\zeta_{k,\infty}^{(\alpha,\beta)})}+\frac{3\zeta_{k,\infty}^{(\alpha,\beta)}}{1-(\zeta_{k,\infty}^{(\alpha,\beta)})^2}+\sum_{{j=1}\atop{j\neq k}}^{|\lambda+|\mu|}\frac{1}{\zeta_{k,\infty}^{(\alpha,\beta)}-\zeta_{j,\infty}^{(\alpha,\beta)}}
\end{align*}
Now divide by $n$ and send $\nri$.  Every term on the right-hand side tends to $0$ because of the simplicity assumption on the zeros of $\Omega_{\lambda,\mu}^{(\alpha,\beta)}$ and \cite[Lemma 2.8]{Bonn}.  In the first sum on the left-hand side, we may apply Lemma \ref{xjacdist} to evaluate the limit.  In the second sum on the left-hand side, every term tends to $0$ except possibly the one for which $j=k$.  Thus
\[
\left(\lim_{\nri}n[\zeta_{k,n}^{(\alpha,\beta)}-\zeta_{k,\infty}^{(\alpha,\beta)}]\right)^{-1}=\lim_{\nri}\frac{1}{n}\sum_{j=1}^{n-|\lambda|-|\mu|}\frac{1}{\zeta_{k,\infty}^{(\alpha,\beta)}-z_{j,n}^{(\alpha,\beta)}}=\sqrt{\frac{1}{\left(\zeta_{k,\infty}^{(\alpha,\beta)}\right)^2-1}}
\]
by Lemma \ref{xjacdist}
\end{proof}

\noindent\textit{Remark.}  Our proof of Theorem \ref{jacthm} used the fact that each point $\{\zeta_{k,\infty}^{(\alpha,\beta)}\}_{k=1}^{|\lambda|+|\mu|}$ attracts exactly one exceptional zero of $P_{\lambda,\mu,n}^{(\alpha,\beta)}$.  This follows from the fact that the number of zeros of $\Omega_{\lambda,\mu}^{(\alpha,\beta)}$ (which we assume to be distinct) matches the number of exceptional zeros of $P_{\lambda,\mu,n}^{(\alpha,\beta)}$ for large $n$.  Our proof yields the same conclusion for any $\zeta_{k,\infty}^{(\alpha,\beta)}$ that is a simple zero of $\Omega_{\lambda,\mu}^{(\alpha,\beta)}$ outside $[-1,1]$ and which has a single nearby zero of $P_{\lambda,\mu,n}^{(\alpha,\beta)}$ for large $n$.

\section{Proof of Theorems \ref{lagthm} and \ref{lag3thm}}\label{lag}

\subsection{Proof in the case $\lambda$ is an even partition}\label{lameven}

For this section, we will retain all of the notation from Section \ref{xlag}.  Additionally, let $\{z_{j,n}^{(\alpha)}\}_{j=1}^{n-|\lambda|-|\mu|}$ denote the regular zeros of $L_{\lambda,\mu,n}^{(\alpha)}$.  We begin with the following lemma.

\begin{lemma}\label{xlagnozero}
Under the hypotheses of Theorem \ref{lagthm}, the polynomials $L_{\lambda,\mu,n}^{(\alpha)}$ and $\Omega_{\lambda,\mu}^{(\alpha)}$ have no zeros in common.
\end{lemma}

\begin{proof}
We will rely on calculations from the proof of \cite[Lemma 11]{BK18}.  In particular, it is shown there that the function
\[
G_n(x):=x^{\alpha+r+1/2}e^{-x^2/2}\frac{L_{\lambda,\mu,n}^{(\alpha)}(x^2)}{\Omega_{\lambda,\mu}^{(\alpha)}(x^2)}
\]
is an eigenfunction of the differential operator $T_{\lambda,\mu}[y]=-y''+V_{\lambda,\mu}y$, where
\[
V_{\lambda,\mu}(x)=x^2+2r+\frac{4(\alpha+r)^2-1}{4x^2}-2\frac{d^2}{dx^2}\log\left(\Omega_{\lambda,\mu}^{(\alpha)}(x^2)\right)
\]
From our assumptions about the zeros of $\Omega_{\lambda,\mu}^{(\alpha)}(x)$, it follows that $V_{\lambda,\mu}(x)$ has a pole of order $2$ at any zero of $\Omega_{\lambda,\mu}^{(\alpha)}(x^2)$.  However, if $G_n(x)$ does not have a pole at a particular zero of $\Omega_{\lambda,\mu}^{(\alpha)}(x^2)$, then neither does $G_n''(x)$ and hence $G_n(x)$ cannot be an eigenfunction of $T_{\lambda,\mu}$.  This contradiction proves the result.
\end{proof}

We will also need the following lemma.

\begin{lemma}\label{xlagdist}
For each $k=1,\ldots,|\lambda+|\mu|$ it holds that
\[
\lim_{\nri}\frac{1}{\sqrt{n}}\sum_{j=1}^{n-|\lambda|-|\mu|}\frac{1}{\zeta_{k,\infty}^{(\alpha)}-z_{j,n}^{(\alpha)}}=\frac{-1}{\sqrt{-\zeta_{k,\infty}^{(\alpha)}}},
\]
where we take the branch of the square root that is positive on $(0,\infty)$ and with branch cut $(-\infty,0]$.
\end{lemma}

The proof of this lemma requires the following result.

\begin{lemma}\label{sjtj}
Suppose that for each $n\in\bbN$ sufficiently large, $\{\rho_{j,n}\}_{j=1}^n$ is a collection of $n$ distinct real numbers with $\rho_{1,n}<\rho_{2,n}<\cdots<\rho_{n,n}$.  Suppose $\{s_{j,n}\}_{j=1}^{n-1}$ satisfies
\[
\rho_{j,n}<s_{j,n}<\rho_{j+1,n},\qquad\qquad j=1,2,\ldots,n-1
\]
Define the open set $\Omega$ by
\[
\Omega=\{z:\Imag[z]\neq0,\}\cup\{z:z<\inf_n\rho_{1,n}\}\cup\{z:z>\sup_n\rho_{n,n}\}
\]
If $x\in\Omega$, then
\begin{equation}\label{tsum}
\left|\sum_{j=1}^n\frac{1}{x-\rho_{j,n}}-\sum_{j=1}^{n-1}\frac{1}{x-s_{j,n}}\right|=O(1)
\end{equation}
as $\nri$ and the error term is uniform on compact subsets of $\Omega$.
\end{lemma}

\begin{proof}
Fix some compact set $K\subseteq\Omega$ and choose $T>0$ so that $\Real[z]\in[1-T,T-1]$ for all $z\in K$.  It is clear that we may drop any single term from the sums of interest at a cost of $O(1)$ as $\nri$, so it suffices to consider
\[
\sum_{j=2}^{n}\frac{1}{x-\rho_{j,n}}-\sum_{j=1}^{n-1}\frac{1}{x-s_{j,n}}=\sum_{j=1}^{n-1}\frac{\rho_{j+1,n}-s_{j,n}}{(x-\rho_{j+1,n})(x-s_{j,n})}.
\]
Let $m,k$ be chosen so that $\{s_{j,n}\}_{j=m}^k=\{s_{j,n}\}_{j=1}^{n-1}\cap[-T,T]$ (if this intersection is empty, then we can skip this step).  Notice that each $((x-\rho_{j+1,n})(x-s_{j,n}))^{-1}$ is bounded by some constant $C_K$ uniformly for $x\in K$ and $j=m,\ldots,k$.  Thus
\[
\left|\sum_{j=m+1}^{k-1}\frac{\rho_{j+1,n}-s_{j,n}}{(x-\rho_{j+1,n})(x-s_{j,n})}\right|\leq C_K\sum_{j=m+1}^{k-1}(\rho_{j+1,n}-s_{j,n})\leq C_K\sum_{j=m+1}^{k-1}(s_{j+1,n}-s_{j,n})\leq2C_KT
\]
(if $m+1>k-1$, then just skip this step).  If $s_{j,n}>T$, then there is a constant $C_K'$ so that
\[
\left|\frac{1}{(x-\rho_{j+1,n})(x-s_{j,n})}\right|\leq \frac{C_K'}{\rho_{j+1,n}s_{j,n}}.
\]
Thus (if we set $s_{n,n}=\infty$)
\[
\left|\sum_{j=k+1}^{n-1}\frac{\rho_{j+1,n}-s_{j,n}}{(x-\rho_{j+1,n})(x-s_{j,n})}\right|\leq \sum_{j=k+1}^{n-1}\left(\frac{C_K'}{s_{j,n}}-\frac{C_K'}{\rho_{j+1,n}}\right)\leq \sum_{j=k+1}^{n-1}\left(\frac{C_K'}{s_{j,n}}-\frac{C_K'}{s_{j+1,n}}\right)\leq\frac{C_K'}{T}.
\]
A similar bound applies when summing over the indices $j=1,\ldots,m-1$.  The only terms we have not accounted for are the terms corresponding to $j\in\{m,k\}$, but as we mentioned above, any finite collection of terms contributes only $O(1)$ to the sum as $\nri$, so the desired estimate holds.
\end{proof}

\begin{proof}[Proof of Lemma \ref{xlagdist}]
Recall that our assumption that $\lambda$ is even implies that $\Omega_{\lambda,\mu}^{(\alpha)}$ has no zeros in $[0,\infty)$.  Therefore, we may invoke \cite[Corollary 2]{BK18} to apply Lemma \ref{sjtj} and write
\begin{align*}
\lim_{\nri}\frac{1}{\sqrt{n}}\sum_{j=1}^{n-|\lambda|-|\mu|}\frac{1}{\zeta_{k,\infty}^{(\alpha)}-z_{j,n}^{(\alpha)}}&=\lim_{\nri}\frac{1}{\sqrt{n}}\sum_{j=1}^{n}\frac{1}{\zeta_{k,\infty}^{(\alpha)}-\sigma_{j,n}^{(\alpha+r)}}
\end{align*}
(recall the notation $\sigma_{j,n}^{(\alpha+r)}$ from the discussion before Theorem \ref{lagthm}), where we used the fact that we could add or delete finitely many terms in this sum without changing its limit.  We can rewrite this as
\begin{align*}
\lim_{\nri}\frac{L_{n}^{(\alpha+r)}(\zeta_{k,\infty}^{(\alpha)})'}{\sqrt{n}L_{n}^{(\alpha+r)}(\zeta_{k,\infty}^{(\alpha)})}=\lim_{\nri}\frac{-L_{n-1}^{(\alpha+r+1)}(\zeta_{k,\infty}^{(\alpha)})}{\sqrt{n}L_{n}^{(\alpha+r)}(\zeta_{k,\infty}^{(\alpha)})}=\frac{-1}{\sqrt{-\zeta_{k,\infty}^{(\alpha)}}},
\end{align*}
where we used Theorem \ref{dhmthm}.
\end{proof}

\begin{proof}[Proof of Theorem \ref{lagthm}]
By Lemma \ref{xlagnozero}, we know that $\zeta_{k,n}^{(\alpha)}\neq\zeta_{k,\infty}^{(\alpha)}$ when $n$ is large.  Therefore, we may use \cite[Equation 105]{BK18}, which tells us
\[
\sum_{j=1}^{n-|\lambda|-|\mu|}\frac{1}{\zeta_{k,\infty}^{(\alpha)}-z_{j,n}^{(\alpha)}}+\sum_{j=1}^{|\lambda|+|\mu|}\frac{1}{\zeta_{k,\infty}^{(\alpha)}-\zeta_{j,n}^{(\alpha)}}=\frac{1}{2}-\frac{\alpha+r}{2\zeta_{k,\infty}^{(\alpha)}}+\sum_{{j=1}\atop{j\neq k}}^{|\lambda|+|\mu|}\frac{1}{\zeta_{k,\infty}^{(\alpha)}-\zeta_{j,\infty}^{(\alpha)}}
\]
Now divide by $\sqrt{n}$ and send $\nri$.  Every term on the right-hand side tends to $0$ because of the simplicity assumption on the zeros of $\Omega_{\lambda,\mu}^{(\alpha)}$.  In the first sum on the left-hand side, we may apply Lemma \ref{xlagdist} to evaluate the limit.  In the second sum on the left-hand side, every term tends to $0$ except possibly the one for which $j=k$.  Thus
\[
\left(\lim_{\nri}\sqrt{n}[\zeta_{k,n}^{(\alpha)}-\zeta_{k,\infty}^{(\alpha)}]\right)^{-1}=\lim_{\nri}\frac{1}{\sqrt{n}}\sum_{j=1}^{n-|\lambda|-|\mu|}\frac{1}{\zeta_{k,\infty}^{(\alpha)}-z_{j,n}^{(\alpha)}}=\frac{-1}{\sqrt{-\zeta_{k,\infty}^{(\alpha)}}}
\]
by Lemma \ref{xlagdist}.
\end{proof}

\subsection{Proof in the Type-$III$ case}\label{3case}

Let us continue to retain the notation from Section \ref{xlag}.  In this section, we will consider the zeros of the Type-$III$ $X_m$ Laguerre polynomials $\{L_{m,n}^{III(\alpha)}\}_{n\in\bbN_m^{III}}$, where $-1<\alpha<0$.  Let us denote the regular zeros of $L_{m,n}^{III(\alpha)}(z)$ by $z_{1,n}^{III(\alpha)},z_{2,n}^{III(\alpha)},\ldots,z_{n-m,n}^{III(\alpha)}$ and the exceptional zeros by $\zeta_{1,n}^{III(\alpha)},\ldots,\zeta_{m,n}^{III(\alpha)}$.  

Recall that the polynomial $L_{m,n}^{III(\alpha)}$ satisfies (\ref{xlag3ode}).  Let us write this as $y''+R_n^{III(\alpha)}(x)y'+S_n^{III(\alpha)}(x)y=0$.  We need the following result

\begin{lemma}\label{lag3pole}
If $n$ is large, then $\zeta_{k,n}^{III(\alpha)}$ is not a pole of $S_n^{III(\alpha)}(x)$ for $k=1,2,\ldots,m$.
\end{lemma}

\begin{proof}
Since the zeros of $L_m^{(-\alpha-1)}$ are all in $(0,\infty)$, the convergence of the exceptional zeros of $L_{m,n}^{III(\alpha)}$ to the zeros of $L_m^{(-\alpha-1)}(-x)$ implies $L_{m,n}^{III(\alpha)}(0)\neq0$ for large $n$, which proves the result.
\end{proof}

We also need the following lemma.

\begin{lemma}\label{xlag3dist}
 Let $\{t_{j,n}^{III(\alpha)}\}_{j=1}^{n-1}$ denote the zeros of $L_{m,n}^{III(\alpha)}$ other than $\zeta_{k,n}^{III(\alpha)}$.  It holds that
\[
\lim_{\nri}\frac{1}{\sqrt{n}}\sum_{j=1}^{n-1}\frac{1}{\zeta_{k,n}^{III(\alpha)}-t_{j,n}^{III(\alpha)}}=\frac{-1}{\sqrt{-\zeta_{k,\infty}^{III(\alpha)}}},
\]
where we take the branch of the square root that is positive on $(0,\infty)$ and with branch cut $(-\infty,0]$.
\end{lemma}

The proof of Lemma \ref{xlag3dist} is very similar to the proof of Lemma \ref{xlagdist}, so we omit the details.  The only major difference is that one must apply Lemma \ref{sjtj} with $x=\zeta_{k,n}^{III(\alpha)}$ instead of $\zeta_{k,\infty}^{III(\alpha)}$.  This requires one to invoke the uniformity in the estimate in Lemma \ref{sjtj} and also in Theorem \ref{dhmthm}.  In order to deduce the necessary interlacing property of the zeros, one must invoke \cite[Theorem 5.5]{LLMS}.

\begin{proof}[Proof of Theorem \ref{lag3thm}]
Evaluate (\ref{xlag3ode}) at the point $\zeta_{k,n}^{III(\alpha)}$.  The last term vanishes (by Lemma \ref{lag3pole}) and \cite[Theorem 5.5]{LLMS} implies that the zeros of $L_{m,n}^{III(\alpha)}(x)$ are simple so $L_{m,n}^{III(\alpha)}(\zeta_{k,n}^{III(\alpha)})'\neq0$.  Thus we may divide through by this quantity.  Let $\{t_{j,n}^{III(\alpha)}\}_{j=1}^{n-1}$ denote the zeros of $L_{m,n}^{III(\alpha)}$ other than $\zeta_{k,n}^{III(\alpha)}$.  We deduce
\[
\sum_{j=1}^{n-1}\frac{2}{\zeta_{k,n}^{III(\alpha)}-t_{j,n}^{III(\alpha)}}+\frac{\alpha+1-\zeta_{k,n}^{III(\alpha)}}{\zeta_{k,n}^{III(\alpha)}}-\sum_{j=1}^m\frac{2}{\zeta_{k,n}^{III(\alpha)}-\zeta_{j,\infty}^{III(\alpha)}}=0
\]
Now divide by $\sqrt{n}$ and send $\nri$.  The middle term tends to $0$ and every term in the last sum tends to $0$ except the one for which $j=k$.  Thus
\[
\left(\lim_{\nri}\sqrt{n}[\zeta_{k,n}^{III(\alpha)}-\zeta_{k,\infty}^{III(\alpha)}]\right)^{-1}=\lim_{\nri}\frac{1}{\sqrt{n}}\sum_{j=1}^{n-1}\frac{1}{\zeta_{k,n}^{III(\alpha)}-t_{j,n}^{III(\alpha)}}=\frac{-1}{\sqrt{-\zeta_{k,\infty}^{III(\alpha)}}}
\]
by Lemma \ref{xlag3dist}.
\end{proof}

\section{Proof of Theorem \ref{hermthm}}\label{herm}

In this section, we will retain the notation from Section \ref{xhermpoly}.  Let us consider the zeros of the $X_m$ Hermite polynomials $\{H_{m,n}^{(\lambda)}\}_{n\in\bbN_{\lambda}}$, where $\lambda$ is an even partition and $m=|\lambda|$.  Let us denote the regular zeros of $H_{m,n}^{(\lambda)}(z)$ by $z_{1,n}^{(\lambda)},z_{2,n}^{(\lambda)},\ldots,z_{n-m,n}^{(\lambda)}$ and the exceptional zeros by $\zeta_{1,n}^{(\lambda)},\ldots,\zeta_{m,n}^{(\lambda)}$.  We will assume that the zeros of $H_{\lambda}$ are simple.

Recall that the polynomial $H_{m,n}^{(\lambda)}$ satisfies (\ref{xhermode}).  Let us write this as $y''+R_n^{(\lambda)}(x)y'+S_n^{(\lambda)}(x)y=0$.  We need the following result

\begin{lemma}\label{hermpole}
If $n$ is large, then $\zeta_{k,n}^{(\lambda)}$ is not a pole of $S_n^{(\lambda)}(x)$ for $k=1,2,\ldots,m$.
\end{lemma}

\begin{proof}
Suppose $H_{\lambda}(\zeta_{k,n}^{(\lambda)})=0$ for some $n$ large and $k\in\{1,\ldots,m\}$. The simplicity assumption on the zeros of $H_{\lambda}$ implies that the exceptional zeros of $H_{m,n}^{(\lambda)}$ are simple for large $n$, so $H_{m,n}^{(\lambda)}(\zeta_{k,n}^{(\lambda)})'\neq0$.  Then in (\ref{xhermode}), the second term has a pole at $\zeta_{k,n}^{(\lambda)}$, while neither of the other terms have a pole there, so the left-hand side is not identically zero.  This gives the desired contradiction.
\end{proof}

We also need the following lemma.

\begin{lemma}\label{xhermdist}
 Let $\{t_{j,n}^{(\lambda)}\}_{j=1}^{n-1}$ denote the zeros of $H_{m,n}^{(\lambda)}$ other than $\zeta_{k,n}^{(\lambda)}$.  It holds that
\[
\lim_{\nri}\frac{1}{\sqrt{n}}\sum_{j=1}^{n-1}\frac{1}{\zeta_{k,n}^{(\lambda)}-t_{j,n}^{(\lambda)}}=\frac{\sqrt{-2\left(\zeta_{k,\infty}^{(\lambda)}\right)^2}}{\zeta_{k,\infty}^{(\lambda)}},
\]
where we take the branch of the square root that is positive on $(0,\infty)$ and with branch cut $(-\infty,0]$.
\end{lemma}

\begin{proof}
The exceptional zeros of $H_{m,n}^{(\lambda)}$ other than $\zeta_{k,n}^{(\lambda)}$ tend to limit points other than $\zeta_{k,\infty}^{(\lambda)}$ as $\nri$ (see \cite[Section 2]{KM}).  Thus, the limit in question will not be affected if we ignore those terms, so it suffices to consider
\[
\lim_{\nri}\frac{1}{\sqrt{n}}\sum_{j=1}^{n-m}\frac{1}{\zeta_{k,n}^{(\lambda)}-z_{j,n}^{(\lambda)}}
\]
Since $H_{\lambda}$ has only non-real zeros and only simple zeros (by assumption), it follows that each $\zeta_{j,n}^{(\lambda)}$ is non-real for all large $n$.

Let $\{\eta_{j,k}\}_{j=1}^k$ denote the zeros of the classical Hermite polynomial $H_k(x)$.  According to \cite[Corollary 4.3]{KM} (which is a consequence of \cite[Theorem 3.1]{BD}) we can employ Lemma \ref{sjtj} to write
\begin{align*}
\lim_{\nri}\frac{1}{\sqrt{n}}\sum_{j=1}^{n-m}\frac{1}{\zeta_{k,n}^{(\lambda)}-z_{j,n}^{(\lambda)}}&=\lim_{\nri}\frac{1}{\sqrt{n}}\sum_{j=1}^{n}\frac{1}{\zeta_{k,n}^{(\lambda)}-\eta_{j,n}}
\end{align*}
where we used the fact that we could remove or insert a finite number of terms into these sums without changing the limit (and we also used the uniformity in Lemma \ref{sjtj}).  We can rewrite this last sum as
\begin{align*}
\lim_{\nri}\frac{H_{n}(\zeta_{k,n}^{(\lambda)})'}{\sqrt{n}H_{n}(\zeta_{k,n}^{(\lambda)})}=\lim_{\nri}\frac{2\sqrt{n}H_{n-1}(\zeta_{k,n}^{(\lambda)})}{H_{n}(\zeta_{k,n}^{(\lambda)})}=\frac{\sqrt{-2\left(\zeta_{k,\infty}^{(\lambda)}\right)^2}}{\zeta_{k,\infty}^{(\lambda)}},
\end{align*}
where we used Theorem \ref{dhmthm} and the relations \eqref{hermlag}.
\end{proof}

\begin{proof}[Proof of Theorem \ref{hermthm}]
Evaluate (\ref{xhermode}) at the point $\zeta_{k,n}^{(\lambda)}$.  The last term disappears by Lemma \ref{hermpole} and we have already observed that $H_{m,n}^{(\lambda)}(\zeta_{k,n}^{(\lambda)})'\neq0$.  Thus we may divide through by this quantity.  Let $\{t_{j,n}^{(\lambda)}\}_{j=1}^{n-1}$ denote the zeros of $H_{m,n}^{(\lambda)}$ other than $\zeta_{k,n}^{(\lambda)}$.  We deduce
\[
\sum_{j=1}^{n-1}\frac{2}{\zeta_{k,n}^{(\lambda)}-t_{j,n}^{(\lambda)}}-2\zeta_{k,n}^{(\lambda)}-\sum_{j=1}^m\frac{2}{\zeta_{k,n}^{(\lambda)}-\zeta_{j,\infty}^{(\lambda)}}=0
\]
(compare with \cite[Equation 5.9]{KM}).  Now divide by $\sqrt{n}$ and send $\nri$.  The middle term tends to $0$ and every term in the last sum tends to $0$ except the one for which $j=k$.  Thus
\[
\left(\lim_{\nri}\sqrt{n}[\zeta_{k,n}^{(\lambda)}-\zeta_{k,\infty}^{(\lambda)}]\right)^{-1}=\lim_{\nri}\frac{1}{\sqrt{n}}\sum_{j=1}^{n-1}\frac{1}{\zeta_{k,n}^{(\lambda)}-t_{j,n}^{(\lambda)}}=\frac{\sqrt{-2\left(\zeta_{k,\infty}^{(\lambda)}\right)^2}}{\zeta_{k,\infty}^{(\lambda)}}
\]
by Lemma \ref{xhermdist}.
\end{proof}

\section{Examples}\label{examples}

In this section we provide some calculations that verify the theorems from the previous sections.  For small degree examples, we can make exact computations.  For larger degrees, we rely on numerical results.

\subsection{$X_2$ Jacobi Example}\label{jacex}

The first kind of exceptional Jacobi polynomials to be extensively studied were the $X_m$ Jacobi polynomials (see \cite{DL,GKM09,GMM,Lun,LLS}).  Here we will consider an example involving $X_2$ Jacobi polynomials.  These are given by
\[
P_{2,n}^{(\alpha,\beta)}(x)=c_2P_{\emptyset,(1,1),n}^{(\alpha-2,\beta+2)}(x),\qquad n\geq m,
\]
where $c_2$ is a constant that is made explicit in \cite[Remark 5.2]{Bonn} (see \cite[Section 5.2]{Bonn}).  For convenience, we will assume that $\alpha>-1$ and $\beta>0$.  We will use the formula
\begin{align*}
P_{2,n}^{(\alpha,\beta)}(x)&=\frac{1}{\alpha+n-1}\bigg[\frac{\alpha+\beta+n-1}{2}(x-1)P_2^{(-\alpha-1,\beta-1)}(x)P_{n-3}^{(\alpha+2,\beta)}(x)\\
&\qquad\qquad\qquad\qquad+(\alpha-1)P_2^{(-\alpha-2,\beta)}(x)P_{n-2}^{(\alpha+1,\beta-1)}(x)\bigg]
\end{align*}
from \cite[Section 2]{LLS}.

In this setting, we have
\[
\Omega_{\emptyset,(1,1)}^{(\alpha-2,\beta+2)}(x)=P_2^{(-\alpha-1,\beta-1)}(x)
\]
Let $\zeta_{1,\infty}^{(\alpha,\beta)},\zeta_{2,\infty}^{(\alpha,\beta)}$ be the zeros of $P_2^{(-\alpha-1,\beta-1)}$ and let us assume they are distinct and outside $[-1,1]$.  Let $y_1^{(\alpha,\beta)},y_2^{(\alpha,\beta)}$ be the zeros of $P_2^{(-\alpha-2,\beta)}(x)$.
Then we have
\[
\frac{P_{n-3}^{(\alpha+2,\beta)}(\zeta_{1,n}^{(\alpha,\beta)})}{P_{n-2}^{(\alpha+1,\beta-1)}(\zeta_{1,n}^{(\alpha,\beta)})}=\frac{2(1-\alpha)P_2^{(-\alpha-2,\beta)}(\zeta_{1,n}^{(\alpha,\beta)})}{(\alpha+\beta+n-1)(\zeta_{1,n}^{(\alpha,\beta)}-1)P_2^{(-\alpha-1,\beta-1)}(\zeta_{1,n}^{(\alpha,\beta)})}
\]
Taking $\nri$ shows
\begin{equation}\label{jacexample}
\lim_{\nri}\frac{P_{n-3}^{(\alpha+2,\beta)}(\zeta_{1,n}^{(\alpha,\beta)})}{P_{n-2}^{(\alpha+1,\beta-1)}(\zeta_{1,n}^{(\alpha,\beta)})}=\frac{2(1-\alpha)(\zeta_{1,\infty}^{(\alpha,\beta)}-y_1^{(\alpha,\beta)})(\zeta_{1,\infty}^{(\alpha,\beta)}-y_2^{(\alpha,\beta)})}{(\zeta_{1,\infty}^{(\alpha,\beta)}-1)(\zeta_{1,\infty}^{(\alpha,\beta)}-\zeta_{2,\infty}^{(\alpha,\beta)})\,\lim_{\nri}n(\zeta_{1,n}^{(\alpha,\beta)}-\zeta_{1,\infty}^{(\alpha,\beta)})}
\end{equation}
Using the differentiation formula \eqref{jacdiff} for the Jacobi polynomials, the left-hand side of (\ref{jacexample}) can be rewritten as
\[
\lim_{\nri}\frac{2P_{n-2}^{(\alpha+1,\beta-1)}(\zeta_{1,n}^{(\alpha,\beta)})'}{(\alpha+\beta+n+1)P_{n-2}^{(\alpha+1,\beta-1)}(\zeta_{1,n}^{(\alpha,\beta)})}=\frac{2}{\sqrt{\left(\zeta_{1,\infty}^{(\alpha,\beta)}\right)^2-1}},
\]
where we used the fact that the limiting zero distribution for the polynomial $P_n^{(\alpha+1,\beta-1)}$ is the equilibrium measure on $[-1,1]$.  One can also use the quadratic formula to find closed form formulas for $\zeta_{1,\infty}^{(\alpha,\beta)}$, $\zeta_{2,\infty}^{(\alpha,\beta)}$, $y_1^{(\alpha,\beta)}$, and $y_2^{(\alpha,\beta)}$ to verify that
\[
\frac{(\zeta_{1,\infty}^{(\alpha,\beta)}-y_1^{(\alpha,\beta)})(\zeta_{1,\infty}^{(\alpha,\beta)}-y_2^{(\alpha,\beta)})}{(\zeta_{1,\infty}^{(\alpha,\beta)}-1)(\zeta_{1,\infty}^{(\alpha,\beta)}-\zeta_{2,\infty}^{(\alpha,\beta)})}=\frac{1}{1-\alpha}
\]
Combining this with our previous estimate shows
\[
\lim_{\nri}n(\zeta_{1,n}^{(\alpha,\beta)}-\zeta_{1,\infty}^{(\alpha,\beta)})=\sqrt{\left(\zeta_{1,\infty}^{(\alpha,\beta)}\right)^2-1}
\]
exactly as predicted by Theorem \ref{jacthm}.

\subsection{Type-$I$ $X_2$ Laguerre Example}\label{lagex}

Here we will consider an example involving Type-$I$ $X_2$ Laguerre polynomials, which are defined by
\[
L_{2,n}^{I(\alpha)}(x)=-L_{\emptyset,(2),n}^{(\alpha-1)}(x),\qquad n\geq 2,\quad\alpha>0
\]
(see \cite[Proposition 4]{BK18}).  In this case, we can use the formula
\[
L_{2,n}^{I(\alpha)}(x)=L_2^{(\alpha)}(-x)L_{n-2}^{(\alpha-1)}(x)+L_2^{(\alpha-1)}(-x)L_{n-3}^{(\alpha)}(x)
\]
(see \cite[Section 3]{LLMS} or \cite[Section 5]{BK18}).  

In this setting, we have
\[
\Omega_{\emptyset,(2)}^{(\alpha-1)}(x)=L_2^{(\alpha-1)}(-x)
\]
One can check that
\[
L_2^{(\alpha)}(x)=\frac{x^2}{2}-(\alpha+2)x+\frac{(\alpha+2)(\alpha+1)}{2}
\]
whose roots are
\[
z_{\pm}:=\alpha+2\pm\sqrt{\alpha+2}
\]
Let $\zeta_{1,\infty}^{I(\alpha)}=-\alpha-1+\sqrt{\alpha+1}$ so that $\zeta_{2,\infty}^{I(\alpha)}=-\alpha-1-\sqrt{\alpha+1}$.  Then, according to our above formula, we have
\[
\frac{(\zeta_{1,n}^{I(\alpha)}-z_+)(\zeta_{1,n}^{I(\alpha)}-z_-)}{(\zeta_{1,n}^{I(\alpha)}-\zeta_{1,\infty}^{I(\alpha)})(\zeta_{1,n}^{I(\alpha)}-\zeta_{2,\infty}^{I(\alpha)})}=\frac{-L_{n-3}^{(\alpha)}(\zeta_{1,n}^{I(\alpha)})}{L_{n-2}^{(\alpha-1)}(\zeta_{1,n}^{I(\alpha)})}
\]
Dividing by $\sqrt{n}$ and sending $\nri$ shows
\begin{equation}\label{lagexample}
\frac{(1+\sqrt{\alpha+1})^2-(\alpha+2)}{2\sqrt{\alpha+1}\,\lim_{\nri}\sqrt{n}(\zeta_{1,n}^{I(\alpha)}-\zeta_{1,\infty}^{I(\alpha)})}=\lim_{\nri}\frac{-L_{n-3}^{(\alpha)}(\zeta_{1,n}^{I(\alpha)})}{\sqrt{n}L_{n-2}^{(\alpha-1)}(\zeta_{1,n}^{I(\alpha)})}
\end{equation}
According to Theorem \ref{dhmthm}, the limit on the right-hand side of \eqref{lagexample} is $-(\alpha+1-\sqrt{\alpha+1})^{-1/2}$.  Thus
\[
\lim_{\nri}\sqrt{n}(\zeta_{1,n}^{I(\alpha)}-\zeta_{1,\infty}^{I(\alpha)})=-(\alpha+1-\sqrt{\alpha+1})^{1/2}=-\sqrt{-\zeta_{1,\infty}^{I(\alpha)}}
\]
exactly as predicted by Theorem \ref{lagthm}.

\subsection{Type-$III$ $X_5$ Laguerre Example}\label{lagex}

Here we will consider an example involving Type-$III$ $X_5$ Laguerre polynomials, which are defined by
\[
L_{5,n}^{III(\alpha)}(x)=-nL_{(1,1,1,1,1),\emptyset,n}^{(\alpha-5)}(x),\qquad n>5,\qquad\alpha\in(-1,0)
\]
(see \cite[Proposition 4]{BK18}).  In this case, we can use the formula
\[
L_{5,n}^{III(\alpha)}(x)=xL_{n-7}^{(\alpha+2)}(x)L_5^{(-\alpha-1)}(-x)+6L_{n-6}^{(\alpha+1)}(x)L_6^{(-\alpha-2)}(-x)
\]
(see \cite[Section 5]{LLMS} or \cite[Section 5]{BK18}).  In this setting, the exceptional zeros converge to the zeros of $L_5^{(-\alpha-1)}(-x)$ as $\nri$ (see \cite[Theorem 5.6]{LLMS}).

If we set $\alpha=-2/5$, then we can numerically calculate the roots of all of these polynomials.  We can observe that there is a zero of $L_5^{(-3/5)}(-x)$ that is approximately equal to $-1.00772514594748$ so we will call this $\zeta_{1,\infty}^{III(-2/5)}$.  Using Mathematica, we calculate
\begin{align*}
\sqrt{100}(\zeta_{1,100}^{III(-2/5)}-\zeta_{1,\infty}^{III(-2/5)})&\approx-1.0377\\
\sqrt{150}(\zeta_{1,150}^{III(-2/5)}-\zeta_{1,\infty}^{III(-2/5)})&\approx-1.03007\\
\sqrt{175}(\zeta_{1,175}^{III(-2/5)}-\zeta_{1,\infty}^{III(-2/5)})&\approx-1.0277\\
\sqrt{200}(\zeta_{1,200}^{III(-2/5)}-\zeta_{1,\infty}^{III(-2/5)})&\approx-1.02584
\end{align*}
We also calculate
\[
-\sqrt{-\zeta_{1,\infty}^{III(-2/5)}}\approx-1.00386
\]
so we appear to be observing the convergence predicted by Theorem \ref{lag3thm}.

\subsection{Hermite Examples}\label{hermex}

Consider the case in which $\lambda=\{1,1\}$ (also studied in \cite[Section 3.1]{CG20}).  This is an even partition, $r=2$, and $m=2$.  One calculates $H_{\lambda}=4(2x^2+1)$, so the zeros of $H_{\lambda}$ are $\pm i/\sqrt{2}$.  Note that they are simple and non-real.  One can now calculate
\[
H_{2,n}^{(\{1,1\})}(x)=\det\begin{pmatrix}
2x & 4x^2-2 & H_n(x)\\
2 & 8x & H_n'(x)\\
0 & 8 & H_n''(x)
\end{pmatrix}
\]
By expanding this determinant and using the fact that $H_n'=2nH_{n-1}$ and the fact that $H_n(x)=2xH_{n-1}(x)-H_{n-1}'(x)$, we find
\[
H_{2,n}^{(\{1,1\})}(x)=16(n-1)\left(-2xH_{n-1}(x)+(n(2x^2+1)-2)H_{n-2}(x)\right)
\]
Let $\zeta_{1,\infty}^{(\{1,1\})}=i/\sqrt{2}$.  Then we have
\[
\frac{H_{n-2}(\zeta_{1,n}^{(\{1,1\})})}{H_{n-1}(\zeta_{1,n}^{(\{1,1\})})}=\frac{\zeta_{1,n}^{(\{1,1\})}}{n(\zeta_{1,n}^{(\{1,1\})}-\zeta_{1,\infty}^{(\{1,1\})})(\zeta_{1,n}^{(\{1,1\})}-\zeta_{2,\infty}^{(\{1,1\})})-1}
\]
This implies
\begin{equation}\label{hermexample}
\lim_{\nri}\frac{\sqrt{n}H_{n-2}(\zeta_{1,n}^{(\{1,1\})})}{H_{n-1}(\zeta_{1,n}^{(\{1,1\})})}=\frac{i/\sqrt{2}}{i\sqrt{2}\lim_{\nri}\sqrt{n}(\zeta_{1,n}^{(\{1,1\})}-\zeta_{1,\infty}^{(\{1,1\})})}=\frac{1}{2\lim_{\nri}\sqrt{n}(\zeta_{1,n}^{(\{1,1\})}-\zeta_{1,\infty}^{(\{1,1\})})}
\end{equation}
According to Theorem \ref{dhmthm} and the relations \eqref{hermlag}, the limit on the far left-hand side of \eqref{hermexample} evaluates to
\[
\frac{\sqrt{-2\left(\zeta_{1,\infty}^{(\{1,1\})}\right)^2}}{2\zeta_{1,\infty}^{(\{1,1\})}}
\]
Thus, the limit in the denominator on the far right-hand side of \eqref{hermexample} is $\zeta_{1,\infty}^{(\{1,1\})}/\sqrt{-2\left(\zeta_{1,\infty}^{(\{1,1\})}\right)^2}$, exactly as predicted by Theorem \ref{hermthm}.

\medskip

As a second example, consider the case when $\lambda=\{2,2\}$.  One calculates $H_{\lambda}=8(4x^4+3)$ so the four zeros of $H_{\lambda}$ are the fourth roots of $-3/4$.  Note that they are simple and non-real.  One can now calculate
\[
H_{2,n}^{(\{2,2\})}(x)=\det\begin{pmatrix}
4x^2-2 & 8x^3-12x & H_{n-2}(x)\\
8x & 24x^2-12 & H_{n-2}'(x)\\
8 & 48x & H_{n-2}''(x)
\end{pmatrix}
\]
One can then find the roots of this polynomial numerically.  Using Mathematica and considering the root $\zeta_{1,\infty}^{\{2,2\}}=\sqrt[4]{3/4}\,e^{i\pi/4}$, we calculate
\begin{align*}
\sqrt{100}(\zeta_{1,100}^{\{2,2\}}-\zeta_{1,\infty}^{\{2,2\}})&\approx-0.000538702+0.719837i\\
\sqrt{200}(\zeta_{1,200}^{\{2,2\}}-\zeta_{1,\infty}^{\{2,2\}})&\approx-0.000262063+0.713381i\\
\sqrt{400}(\zeta_{1,400}^{\{2,2\}}-\zeta_{1,\infty}^{\{2,2\}})&\approx-0.00012928+0.710222i\\
\sqrt{500}(\zeta_{1,500}^{\{2,2\}}-\zeta_{1,\infty}^{\{2,2\}})&\approx-0.000103149+0.709596i
\end{align*}
We also calculate
\[
\frac{\zeta_{1,\infty}^{(\{2,2\})}}{\sqrt{-2\left(\zeta_{1,\infty}^{(\{2,2\})}\right)^2}}\approx0.707107i
\]
so it appears that we are observing the convergence predicted by Theorem \ref{hermthm}.



\end{document}